\documentclass{article}
\usepackage[english]{babel}
\usepackage[utf8x]{inputenc}
\usepackage[T1]{fontenc}
\usepackage[a4paper,top=2cm,bottom=2cm,left=3cm,right=3cm,marginparwidth=1.75cm]{geometry}
\usepackage{graphicx}
\usepackage{float}
\usepackage{amsmath,amssymb,amsthm}
\usepackage{bm}
\usepackage[dvipsnames]{xcolor}
\usepackage{tikz}
\usepackage{comment}
\usepackage[numbers,sort]{natbib}
\usepackage{algorithm,algorithmic,float}
\usepackage{enumitem}
\usepackage{comment}
\usepackage{mathtools}
\usepackage{url}
\usepackage[textsize=tiny,colorinlistoftodos]{todonotes}
\usepackage{marginnote}
\usepackage{etoolbox}

\usepackage{hyperref}
\hypersetup{
    colorlinks,
    breaklinks=true,
    linkcolor={blue},
    citecolor={blue},
    urlcolor={blue}
}

\usepackage{cleveref}
\usepackage{times}
\usepackage{multirow}
\usepackage{makecell}
\usepackage{siunitx}
\usepackage{longtable}
\usepackage[normalem]{ulem}
\usepackage[most]{tcolorbox}
\usepackage{cancel}

\newtcolorbox{mainprob}[1][]{
  fonttitle=\bfseries,
  arc=3mm,
  boxrule=0.8pt,
  title=Main problem\ifx\\#1\\\else\ (#1)\fi
}

\usepackage{booktabs}

\newtheorem{lemma}{Lemma}[section]
\newtheorem{theorem}[lemma]{Theorem}

\newtheorem{proposition}[lemma]{Proposition}
\newtheorem{corollary}[lemma]{Corollary}

\theoremstyle{definition}
\newtheorem{definition}[lemma]{Definition}
\newtheorem{example}[lemma]{Example}
\newtheorem{remark}[lemma]{Remark}
\newtheorem{notation}[lemma]{Notation}

\AddToHook{env/theorem/begin}{\crefalias{lemma}{theorem}}
\AddToHook{env/conjecture/begin}{\crefalias{lemma}{conjecture}}
\AddToHook{env/question/begin}{\crefalias{lemma}{question}}
\AddToHook{env/proposition/begin}{\crefalias{lemma}{proposition}}
\AddToHook{env/corollary/begin}{\crefalias{lemma}{corollary}}

\AddToHook{env/definition/begin}{\crefalias{lemma}{definition}}
\AddToHook{env/example/begin}{\crefalias{lemma}{example}}
\AddToHook{env/remark/begin}{\crefalias{lemma}{remark}}
\AddToHook{env/notation/begin}{\crefalias{lemma}{notation}}

\crefname{lemma}{Lemma}{Lemmas}
\crefname{theorem}{Theorem}{Theorems}
\crefname{conjecture}{Conjecture}{Conjectures}
\crefname{question}{Question}{Questions}
\crefname{proposition}{Proposition}{Propositions}
\crefname{corollary}{Corollary}{Corollaries}
\crefname{definition}{Definition}{Definitions}
\crefname{example}{Example}{Examples}
\crefname{remark}{Remark}{Remarks}
\crefname{notation}{Notation}{Notations}

\newcommand{\CC}{\mathbb{C}}

\newcommand{\QQ}{\mathbb{Q}}

\newcommand{\ZZpos}{\mathbb{N} \cup \{0,-1\}}

\newcommand{\bx}{\mathbf{x}}  
\newcommand{\bz}{\mathbf{z}}

\newcommand{\bc}{\mathbf{c}}  
\newcommand{\by}{\mathbf{y}}  
\newcommand{\bh}{\mathbf{h}}  
\newcommand{\bp}{\mathbf{p}}

\newcommand{\bq}{\mathbf{q}} 
\newcommand{\bF}{\mathbf{f}} 
\newcommand{\bg}{\mathbf{g}} 
\newcommand{\bu}{\mathbf{u}} 
 
\newcommand{\bM}{\mathbf{M}} 
\newcommand{\bU}{\mathbf{U}} 
\newcommand{\bX}{\mathbf{X}} 
\newcommand{\bY}{\mathbf{Y}} 

\newcommand{\bmu}{\boldsymbol{\mu}}

\newcommand{\cF}{\mathcal{F}}
\newcommand{\cK}{\mathcal{K}}
\newcommand{\cE}{\mathcal{E}}
\newcommand{\cL}{\mathcal{L}}

\makeatletter
\newcommand*\rel@kern[1]{\kern#1\dimexpr\macc@kerna}
\newcommand*\widebar[1]{%
  \begingroup
  \def\mathaccent##1##2{%
    \rel@kern{0.8}%
    \overline{\rel@kern{-0.8}\macc@nucleus\rel@kern{0.2}}%
    \rel@kern{-0.2}%
  }%
  \macc@depth\@ne
  \let\math@bgroup\@empty \let\math@egroup\macc@set@skewchar
  \mathsurround\z@ \frozen@everymath{\mathgroup\macc@group\relax}%
  \macc@set@skewchar\relax
  \let\mathaccentV\macc@nested@a
  \macc@nested@a\relax_111{#1}%
  \endgroup
}
\makeatother

\usetikzlibrary{cd} 
\tikzset{
    lablvert/.style={anchor=south, rotate=90, inner sep=.5em}
}

\DeclareMathOperator{\ord}{ord}

\title{Observable functions of rational ODE models\\and how to find them\thanks{AD has been supported by an ERC-2023-ADG grant for the ODELIX project (number 101142171) and partially supported by the NSF grant CCF-2212460. GP has been supported by  the French ANR-22-CE48-0008 OCCAM project.}}

\author{Alexander Demin\thanks{
Laboratoire d'informatique de l'École polytechnique (LIX, UMR 7161), CNRS, École polytechnique, Institut Polytechnique de Paris, Palaiseau, France}, 
Gleb Pogudin\footnotemark[2], 
Christopher Rackauckas\thanks{JuliaHub Inc., Massachusetts Institute of Technology}}
\date{\today}

\begin{document}

\maketitle

\makeatletter
\begingroup
\renewcommand{\thefootnote}{}
\renewcommand{\@makefntext}[1]{\noindent #1}

\footnotetext{%
\vspace*{0.6em}%
\begin{minipage}[t]{0.85\linewidth}
\vspace{0pt}%
\itshape
Funded by the European Union. Views and opinions expressed are however
those of the author(s) only and do not necessarily reflect those of the
European Union or the European Research Council Executive Agency.
Neither the European Union nor the granting authority can be held
responsible for them.
\end{minipage}%
\hfill
\begin{minipage}[t]{0.13\linewidth}
\vspace{0.20em}
\includegraphics[height=0.80cm]{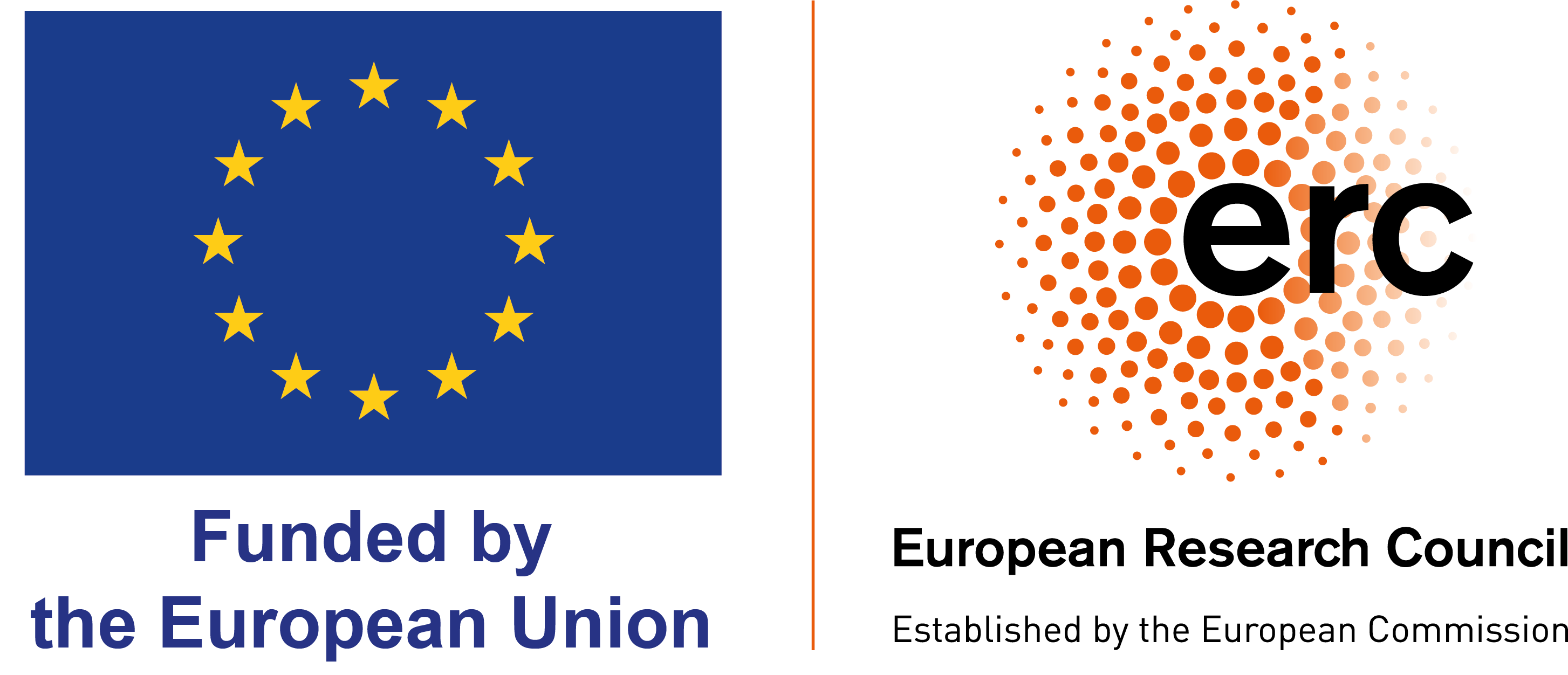}
\end{minipage}%
}
\endgroup
\makeatother

\begin{abstract}



Consider a parametric ODE control model.
A function of the
states and parameters is called \emph{observable} if its value can in principle be reconstructed from input-output data.
The observable functions form a field, called the \emph{observation field}, represented naturally by a set of generators.
Even when the model is not fully observable, this field captures the information still accessible from input-output data.

We present an algorithm for computing a concise generating set for the observation field of a model with rational dynamics. The algorithm relies on two new results: one allows observable functions to be extracted from the coefficients of repeated Lie derivatives of the outputs, while the other reduces the required orders of differentiation by exploiting identifiable parameter combinations.

We implement the resulting algorithm in
\href{https://github.com/SciML/StructuralIdentifiability.jl}
{StructuralIdentifiability.jl}. For computational efficiency, we employ recent techniques for differential elimination and rational function field simplification. 
Using models from epidemiology, chemical kinetics, and cancer modeling, we show that the algorithm produces generators with domain-specific interpretations that can inform model analysis and development.
\end{abstract}

\section{Introduction}

Consider a dynamical model given by a system of ordinary differential equations:
\begin{equation*}\label{eq:ODEmodel-1}
\Sigma = \left\{\begin{aligned}
\bx'(t) = \bF(\bmu, \bx(t), \bu(t)),\\    
\by(t) = \bg(\bmu, \bx(t), \bu(t)),
\end{aligned}\right.
\end{equation*}
where $\bx = (x_1,\ldots,x_n)$ are state variables, $\bmu = (\mu_1,\ldots,\mu_k)$ are time-independent parameters, $\bu = (u_1,\ldots,u_\ell)$ are inputs, and $\by = (y_1,\ldots,y_m)$ are outputs. 
Roughly speaking, structural global observability (in what follows, simply ``observability'') asks whether the state can be uniquely recovered from the known input and observed output under ideal conditions~\cite{KOU197389,HermannKrener}. 
Identifiability is the analogous property for the time-independent parameters of the model~\cite{BELLMAN1970329,Hong20}.

In practice, models frequently lack observability and/or identifiability~\cite{Massonis2021,Mahdi2014,Chatzis2014}.
Yet this need not render such models
completely opaque: 
even when individual states and parameters
are not observable, their combinations might be.
In particular, a function $h(\bmu, \bx)$ is said to be observable (or functionally observable~\cite{Montanari2022,Kravaris2024}) if its value can in principle be reconstructed from 
input-output data.
Together, such functions form
the \emph{observation field}~\cite{observability-algebra,Bartosiewicz1987} of the model,
\begin{equation*}\label{eq:intro-observations}
\cF_\Sigma := \{ h \in \CC(\bmu, \bx) \mid h \text{ is observable}\},
\end{equation*}
which is a nonlinear analogue of
the classical notion of an observation space for linear systems~\cite{FLIESS1986147}.
Besides observable states and parameters, $\cF_\Sigma$
 may contain some other functions of practical interest 
(e.g., the basic reproduction number~\cite{Pant2026} or the binding density~\cite[Section~3.3]{ilmer21}).
Furthermore, knowledge of
generators of $\cF_\Sigma$
can guide system identification~\cite{Norden2026}, 
the selection of complementary data streams~\cite{Pant2026}, and model reparametrization~\cite{Nemcova2016}.
In this paper, we consider the problem of finding a concise 
generating set of
the observation field $\cF_\Sigma$ for a given rational system (that is, when $\bf f$ and $\bg$ are rational functions).

From a
computational point of view, algorithms and software are available for testing
observability of state variables~\cite{Saccomani19,Hong20}. 
Furthermore, N\u{e}mcov\`a, Petreczky, and van Schuppen~\cite{Nemcova2016} proposed (building upon earlier works~\cite{NEMCOVA2012953,Bartosiewicz1987,Bartosiewicz1988}) a complete and rigorous procedure 
for computing generators of the observation field of                                         input-affine rational systems
from
sufficiently many Lie derivatives of the outputs.
The procedure is constructive but relies on heavy computational tools from commutative algebra and, to the best of our knowledge, has not
been implemented. 
A related problem of finding \emph{locally} observable functions~\cite{Evans2000,Stigter2015,Villaverde23,borgqvist2026framinglocalstructuralidentifiability} has typically been reduced 
to a search for analytic solutions of certain PDEs or ODEs, which itself is a hard problem with no complete algorithmic solution available.
Consequently, these methods do not guarantee to find a generating set (see also~\cite[Corollary~3]{balakrishnan2022datadrivenobservabilitydecompositionkoopman}).

The main result of the present paper is an algorithm for computing generators of the observation field for models with rational dynamics.
In particular, we allow arbitrary rational dependence on inputs.
Our algorithm is designed to be practical and we provide an implementation within the Julia package~\href{https://github.com/SciML/StructuralIdentifiability.jl}
{StructuralIdentifiability.jl}.
We apply it to models from the literature arising in epidemiology, chemistry, and cancer modeling, and show that it is efficient and produces generators that admit natural model-specific interpretations.
As in previous works~\cite{Hong20,Nemcova2016}, 
repeated Lie derivatives of the outputs 
serve as one of the building blocks in our algorithm, which we combine with the following two key technical contributions.

The 
first technical contribution allows us to extract observable functions from 
the Lie derivatives
even when
the dependence on the inputs is not affine.
More precisely, we show that if an observable function is written as a differential rational function in $\bu$ with coefficients in $\CC(\bmu,\bx)$, then all coefficients of its numerator and denominator are themselves observable
(\Cref{thm:split_input}).
This can be viewed as an extension of the classical result for input-affine systems due to Sontag and Wang~\cite{WANG1989279}.
Alternatively, one could reduce the problem to the input-affine case via state augmentation.
However, this would not only increase the dimension and hence the
computational cost (see~\cite{Martinelli2026} for a related discussion), but also introduce the artificial state variables that would have to be filtered from the output of the algorithm.

Still, high-order Lie derivatives tend to be large expressions jeopardizing the efficiency and interpretability.
Our second contribution reduces the required orders
of Lie derivatives. 
We show that instead of generating the observation field from scratch using repeated Lie derivatives, one may start with some known observable functions
and complete them with the Lie derivatives of lower order (\Cref{theorem:complete}). 
We propose to obtain such a starting set using identifiable functions, that is, observable functions depending only on the parameters, which can be read off from the input-output equations of the model~\cite{OllivierThesis,Ovchinnikov21}.

Our implementation builds on a recent efficient approach to computing input-output equations~\cite{Dong2023} and 
on our prior work on field generators simplification~\cite{demin2026simplegeneratorsrationalfunction}.
The latter simplification step turns out to be instrumental in making the result of the computation concise and interpretable.

The rest of the paper is organized as follows.
\Cref{sec:preliminaries} contains the necessary background and introduces the problem formally.
In \Cref{sec:informal}, we give an informal overview of the main ideas illustrating them with a specific example.
\Cref{sec:formal} contains the main theoretical results on input separation
and completion of the observation field and the description of the algorithm.
\Cref{sec:implementation} describes our implementation of the algorithm and demonstrates the impact of the aforementioned improvements on its efficiency.
Finally, \Cref{sec:examples} contains computational experiments and case studies involving models from the literature.

\paragraph{Acknowledgements.} We would like to thank Oren Bassik, Hoon Hong, Alexey Ovchinnikov, Binod Pant, and Alejandro F. Villaverde for helpful discussions.


\section{Preliminaries and problem statement}\label{sec:preliminaries}

Consider an ODE system with control in the state-space form:
 \begin{equation}\label{eq:ODEmodel}
   \Sigma = \begin{cases}
     \bx'(t) = \bF(\bmu, \bx(t), \bu(t)),\\
     \by(t) = \bg(\bmu, \bx(t), \bu(t)),
    \end{cases}
   \end{equation}
where: $\bx(t)$ are state variables, $\bmu$ are time-independent parameters, $\bu(t)$ are input variables, $\by(t)$ are output variables, and $\bF$ and $\bg$ are 
vectors
of elements of $\CC(\bmu, \bx(t), \bu(t))$, that is, rational functions in $\bmu, \bx(t), \bu(t)$.
In what follows, we may, for the sake of brevity, omit the explicit time dependence for $\bx, \by, \bu$.

The central notion of this paper is \emph{observability} which is usually defined for the state variables or functions of them.
On the other hand, for functions in parameters one typically uses term \emph{identifiability}.
In our context, this distinction is artificial, and we are interested in considering combinations potentially involving both states and parameter.
We will, therefore, use the term 
observability
for states and parameters.

In this section we will give a formal definition of 
observability
 using algebraic language which will be more convenient in our case.
This definition is mathematically equivalent to the standard analytic one~\cite[Proposition~3.4]{Hong20} (see also~\cite[Definition~2.5]{Ovchinnikov21}).

\begin{definition}[Differential rings and fields]
      A \emph{differential ring} $(R,\,')$ is a commutative ring with a derivation $'\colon R\to R$, that is, a map such that, for all $a, b \in R$, $(a + b)' = a' + b'$ and $(ab)' = a' b + a b'$ (Leibniz rule). 
      A {\em differential field} is a differential ring that is a field.
      For  
      $i \geqslant 0$,
      $a^{(i)}$ denotes the $i$-th order derivative of $a \in R$.
\end{definition}

\begin{example}
    \begin{itemize}
        \item[]
        \item The sets of smooth or analytic functions on an interval are natural examples of differential rings with respect to the standard derivation.
        The set of functions meromorphic on a domain is a differential field.

        \item Any polynomial differential equation in an unknown function $x$ can be represented as an element of an infinite-dimensional polynomial ring $\CC[x, x', x'', \ldots]$ in formal variables $x, x', x'', \ldots$ which has a natural structure of a differential ring with the derivation defined as $(x^{(i)})' = x^{(i + 1)}$ for all $i \geqslant 0$.
        Analogously, one can define the field $\mathbb{C}(x, x', x'', \ldots)$ with the derivation $(x^{(i)})' = x^{(i + 1)}$.
    \end{itemize}
\end{example}

\begin{notation}
  Let $x$ be an element of a differential ring (e.g., a smooth function or a differential indeterminate) and $h \in \mathbb{Z}_{\geqslant 0}$. We introduce
  \[
      x^{(<h)} := (x, x', \ldots, x^{(h - 1)}),\quad
      x^{(\infty)} := (x, x', x'', \ldots).
  \]
  $x^{(\leqslant h)}$ is defined analogously.
  If $\bx = (x_1, \ldots, x_n)$ is a tuple of elements of a differential ring, then
  \[
      \bx^{(< h)} := (x_1^{(< h)}, \ldots, x_n^{(<h)}),\quad 
      \bx^{(\infty)} := (x_1^{(\infty)}, \ldots, x_n^{(\infty)}).
  \]
  In particular, $\CC[x^{(\infty)}]$ and $\CC(x^{(\infty)})$ will denote the ring of polynomials and the field of rational functions in $x$ and its derivatives, respectively. Elements of this ring (resp., field) are polynomials (resp.,  rational functions) in finitely many of the $x^{(\infty)}$.
\end{notation}

\begin{definition}[Observability]
    \label{def:identifiability2}
    A function $h \in \mathbb{C}(\bmu, \bx, \bu^{(\infty)})$ is called \emph{observable} if there exist polynomials $P_1, P_2 \in \mathbb{C}[\by^{(\infty)}, \bu^{(\infty)}]$ such that $P_2$ does not vanish identically on all the analytic solutions of~\eqref{eq:ODEmodel} and, for every analytic solution $\bmu^\ast, \bx^\ast, \bu^\ast, \by^\ast$ of~\eqref{eq:ODEmodel} on which $P_2$ does not vanish, we have
    \begin{equation}\label{eq:observ_def}
        h(\bmu^*, \bx^*, \bu^*) = \frac{P_1(\by^*, \bu^*)}{P_2(\by^*, \bu^*)}.
    \end{equation}
\end{definition}

\begin{remark}
    Intuitively, \Cref{def:identifiability2}
    corresponds to the possibility (under the assumption of complete and noise-free data) to reconstruct the value of the function in question.
    On one hand, if the quality of the data allows to estimate high-order derivatives of $\by^*$ and $\bu^*$, the equality~\eqref{eq:observ_def} already provides a theoretical approach to such estimation.
    But furthermore, as has been shown in~\cite[Proposition~3.4]{Hong20} (and extended in~\cite[Proposition~3.1]{Ovchinnikov21}), the possibility to reconstruct the value by any means (as given by the analytic definition~\cite[Definition~2.3]{Hong20}, \cite[Definition~2.5]{Ovchinnikov21}) always implies the existence of a formula~\eqref{eq:observ_def}.
\end{remark}

The set of all observable functions is an important object describing what we in principle can know about the state of the system.
Since the set of observable functions is closed under arithmetic operations (one can simply perform the corresponding arithmetic operations on the right-hand sides of~\eqref{eq:observ_def}), it forms a field.
One typically considers not this whole field but its restriction to states and parameters, thus excluding the inputs, which leads to the following definition.

\begin{definition}[Observation field~\cite{SontagWangIO,Bartosiewicz1987}]
\label{def:field_of_ident}
    For a model~\eqref{eq:ODEmodel}, the set
    \begin{equation}\label{eq:obs_field_def}
    \mathcal{F}_\Sigma := \{ h \in \CC(\bmu, \bx) \mid h \text{ is observable}\}
    \end{equation}
    is called the~\emph{observation field} of the model.
\end{definition}

It may seem that, by restricting ourselves to $\CC(\bmu, \bx)$ in~\eqref{eq:obs_field_def} we may lose some interesting observable
functions from $\CC(\bmu, \bx, \bu^{(\infty)})$.
As we shall see, this is not the case because all observable
functions can be expressed in terms of $\mathcal{F}_{\Sigma}$ and $\bu^{(\infty)}$ (\Cref{thm:split_input}).

\begin{example}
Whenever
the states of the model are observable and its parameters are identifiable, the observation field coincides with the whole field of rational functions, that is, $\mathcal{F}_{\Sigma} = \CC(\bmu, \bx)$.
\end{example}

\begin{example}
    Consider a model with a single state $x'(t) = \mu_1 x(t)$ and single output $y(t) = \mu_2 x(t)$.
    The equation admits explicit solution $x(t) = x_0 e^{\mu_1 t},\; y(t) = \mu_2 x_0 e^{\mu_1 t}$.
    Therefore, if $\mu_2 x_0 \neq 0$, the knowledge of $y(t)$ and its derivatives allows to reconstruct only $\mu_1$ and $\mu_2 x_0$ but not $x_0$.
    In terms of \emph{structural} observability, it is possible to reconstruct $\mu_1$ and $\mu_2 x(t)$ but not $x(t)$.
    Furthermore, every observable function is a combination of these, since the output is unchanged under the symmetry $x \mapsto \lambda x, \mu_2 \mapsto \frac{\mu_2}{\lambda}$ for every $\lambda \neq 0$.
    Thus, the observation field is generated by $\mu_1$ and $\mu_2 x(t)$.
\end{example}

\begin{remark}[On other definitions]
    A different definition of the observation field, as the field containing the outputs and invariant under the family of vector fields obtained from $\bF$
    via evaluation of $\bu$, also appears in the literature (e.g.~\cite[Definition III.3]{Nemcova2016}, \cite[Definition~4]{Bartosiewicz1987}).
    Our \Cref{thm:split_input} implies that it is equivalent to~\Cref{def:field_of_ident}.
\end{remark}


\begin{mainprob}
\begin{description}
  \item[Given:] 
  an ODE model $\Sigma$, as in~\eqref{eq:ODEmodel}, given by $ \bx' = \bF(\bmu, \bx, \bu)$ and $\by = \bg(\bmu, \bx, \bu)$
  \item[Find:] a set of generators $h_1,\ldots, h_\ell\in \mathbb{C}(\bmu, \bx)$ of the 
  observation field $\mathcal{F}_\Sigma$
\end{description}
\end{mainprob}

\begin{remark}
    Different solutions to the main problem may not be equally concise and useful.
    In the next section, for the
    example~\eqref{eq:running}, we will construct three different sets of generators: \eqref{eq:gens_Lie_only}, \eqref{eq:gens_io}, and \eqref{eq:gens_final}.
    The first one, \eqref{eq:gens_Lie_only}, contains seven generators of degrees up to eleven, while 
    the final one, \eqref{eq:gens_final}, has only four generators of degree two, making it potentially more amenable to interpretation.
    
\end{remark}


\section{Informal description of solution}
\label{sec:informal}

In this section we present a high-level informal description of our approach to computing generators of the observation field. 
To keep notation simple, we consider system as in~\eqref{eq:ODEmodel} with a single output $y(t)$ and a single input $u(t)$.
The description will be illustrated using the following 
example system $\Sigma$:
\begin{equation}\label{eq:running}
\begin{cases}
    x'(t) = \dfrac{(\mu_1 x(t) + \mu_2)^2}{\mu_3 + u(t)},\\
    y(t) = \mu_4 x(t) + \mu_5
\end{cases}
\end{equation}
We recall that the Lie derivative of a function 
$h(\bmu, \bx, \bu, \bu', \ldots)$
with respect to the vector field~\eqref{eq:ODEmodel} is defined~as 
\[
    \cL_\Sigma(h) := \sum\limits_{i = 1}^n f_i \frac{\partial}{\partial x_i}(h) + \sum\limits_{j = 1}^\ell\sum\limits_{k = 0}^\infty u_j^{(k + 1)} \frac{\partial }{\partial u_j^{(k)}} (h).
\]
It represents the total derivative of $h$ with respect to time.


\paragraph{Generating all observable functions.}

By~\eqref{eq:ODEmodel}, $y$ is a rational function of $\bmu, \bx, u$.
Furthermore, by computing the Lie derivative, one can show that, for any $i \geqslant 0$, $y^{(i)}$ is a rational function of $\bmu, \bx, u, u', \ldots, u^{(i)}$.

Then~\Cref{def:identifiability2} implies that all observable functions can be obtained from $g, \mathcal{L}_{\Sigma}(g), \mathcal{L}_{\Sigma}^2(g), \ldots$ and $u, u', \ldots$ by arithmetic operations.
In other words, these derivatives can be viewed as the Taylor coefficients of the functions $y(t)$ and $u(t)$, respectively, which, in the case of analytic solutions, determine them uniquely~\cite{P1978}.

In fact, general field-theoretic considerations show~\cite[Proposition~1]{Bartosiewicz1987} that it is sufficient to take only finitely many Lie derivatives.
Furthermore, using~\cite[Theorem~3.16]{Hong20}, \cite[Corollary IV.4]{Nemcova2016}, or~\Cref{theorem:complete} one can show that Lie derivatives up to order $|\bx| + |\bmu|$ (inclusive) suffice.
For specific systems, this bound can be refined:
\Cref{theorem:complete} implies that every observable function
in our running example is a rational function of
$u^{(\infty)}$ and of the  
Lie derivatives of the order up to four
(although $|\bx| + |\bmu| = 6$):
\begin{equation}\label{eq:gens_with_u}
\underbrace{\mu_4 x + \mu_5}_{=g},\;\; \underbrace{\mu_4 \frac{(\mu_1 x + \mu_2)^2}{\mu_3 + u}}_{= \mathcal{L}_{\Sigma}(g)},\; \underbrace{\mu_4 \frac{(\mu_1 x + \mu_2)^2}{(\mu_3 + u)^2} (2 \mu_1 (\mu_1 x + \mu_2)-u')}_{=\mathcal{L}_{\Sigma}^2(g)},\;\; 
\underbrace{\mathcal{L}_{\Sigma}^3(g),\;\; \mathcal{L}_{\Sigma}^4(g)}_{\text{too large to display}}.
\end{equation}
Since~\Cref{theorem:complete} is formulated in terms of the existence of an algebraic relation between the Lie derivatives, its condition can be verified using the observability rank condition~\cite{HermannKrener}.

\paragraph{Carving out the observation field.}
The outlined approach does not completely solve the main problem of computing \emph{the observation field} since the produced expressions
~\eqref{eq:gens_with_u} may contain inputs and their derivatives.
If the original system~\eqref{eq:ODEmodel} is input-affine (that is, $\bF(\bmu, \bx, u) = \bF_0(\bmu, \bx) + \bF_1(\bmu, \bx) u$), then the results of Sontag and Wang~\cite[Theorem~1]{WANG1989279} imply that the observation field is generated by the coefficients of the Lie derivatives
viewed as polynomials in $u^{(\infty)}$.
This result is not applicable in~\eqref{eq:gens_with_u}, where dependency on the input is rational.

We extend 
the result of Sontag and Wang
on ODE models with arbitrary rational dependence on inputs: if an observable rational function is written as a rational function in the input and its derivatives $u,u',u'',\ldots$ with coefficients in $\CC(\bmu,\bx)$, then all coefficients of its numerator and denominator are themselves observable (\Cref{thm:split_input}).
This implies, in particular, that the coefficients (with respect to $u^{(\infty)}$) of the Lie derivatives are observable. In the running example, for the first order Lie derivative we get
\[
\mathcal{L}_{\Sigma}(g) = \frac{A}{B + u} \quad\Longrightarrow\quad 
A~\text{and}~B~\text{are observable},
\]
where $A = \mu_4(\mu_1 x + \mu_2)^2$ and $B = \mu_3$.

Furthermore, instead of successive Lie derivatives as in~\eqref{eq:gens_with_u}, one can alternate between taking Lie derivatives and applying \Cref{thm:split_input}.
With this approach,  
fewer differentiation steps may be needed 
(that is, the condition of~\Cref{theorem:complete} would be satisfied earlier).
In the context of our running example, it turns out to be sufficient to take the Lie derivative only up to order three.
This results in the following generators of the observation field:
\begin{equation}\label{eq:gens_Lie_only}
    \underbrace{x \mu_4 + \mu_5}_{=g},\; 
    \underbrace{
    \mu_4(\mu_1 x + \mu_2)^2,\; \mu_3
    }_{\text{coefficients of }\mathcal{L}_\Sigma(g)},\; \underbrace{2 \mu_1 \mu_4(\mu_1 x + \mu_2)^3,\; \mu_3}_{\text{coefficients of }\mathcal{L}_\Sigma(\mu_4(\mu_1x + \mu_2)^2)},
    \underbrace{6\mu_1^2\mu_4(\mu_1 x + \mu_2)^4,\; \mu_3}_{\text{coefficients of }\mathcal{L}_\Sigma(2\mu_1 \mu_4(\mu_1x + \mu_2)^3)}
\end{equation}

\paragraph{Making the computation more efficient and the result more useful.}
The two preceding subsections already yield an algorithm for computing generators of the observation field (reminiscent of the one from~\cite{Nemcova2016}).
However, in practice, high-order Lie derivatives are usually huge symbolic expressions, which makes this approach not very efficient computationally and its output 
too difficult to use.

Our second theoretical result, \Cref{theorem:complete}, provides a sort of scaffolding: it is sufficient to take fewer Lie derivatives given that we have already obtained some observable functions in another way.
This is accomplished by determining the observable functions that depend only on the parameters, i.e., elements of the subfield $\cF_{\Sigma} \cap \CC(\bmu)$. 
For such functions, observability is usually referred to as identifiability. 
One standard approach to this problem uses input-output equations~\cite{OllivierThesis}, that is, differential equations satisfied by $y(t)$ and $u(t)$ with coefficients depending on $\bmu$. 
For our running example, eliminating the states gives
\begin{equation}\label{eq:uwu-io}
    \frac{(\mu_1 \mu_5 - \mu_2 \mu_4)^2}{\mu_4} + 2\mu_1 \frac{\mu_2 \mu_4 - \mu_1 \mu_5}{\mu_4} y + \frac{\mu_1^2}{\mu_4} y^2 - \mu_3 y' - u y' = 0 
\end{equation} 
For a sufficiently exciting input, the known functions $y, y^2, y', uy'$ are linearly independent over constants, so the coefficients in~\eqref{eq:uwu-io} can be uniquely determined. 
For example, this can be done by evaluating~\eqref{eq:uwu-io} at sufficiently many time points and solving the resulting linear system for the coefficients. 
Hence the coefficients are observable.

\Cref{theorem:complete} implies that these coefficients can be complemented to a generating set of the observation field by considering only 
the output and its first order Lie derivative
and eliminating the inputs as described above.
The resulting generators are
\begin{equation}\label{eq:gens_io}
\underbrace{\mu_4 x + \mu_5,\; \mu_3,\; \mu_4 (\mu_1 x + \mu_2)^2}_{\text{from Lie derivatives}},\; \underbrace{\frac{(\mu_1\mu_5 - \mu_2\mu_4)^2}{\mu_4}, \; \frac{\mu_1 (\mu_2 \mu_4 - \mu_1 \mu_5)}{\mu_4},\; \frac{\mu_1^2}{\mu_4}, \mu_3}_{\text{from input-output equation~\eqref{eq:uwu-io}}}
\end{equation}
These expressions are simpler compared to the ones obtained by a more straightforward method~\eqref{eq:gens_Lie_only}: they have lower degree and the role of different parameter combinations is easier to see. However, there is still some room for simplification: for example, the fourth generator can be expressed in terms of the fifth and the sixth.
Therefore, we finish the computation by applying a simplification algorithm developed in our previous work~\cite{demin2026simplegeneratorsrationalfunction} and obtaining:
\begin{equation}\label{eq:gens_final}
\mu_4 x + \mu_5, \; \mu_3, \; \frac{\mu_1^2}{\mu_4}, \; \frac{\mu_1 \mu_5 - \mu_2 \mu_4}{\mu_1}.
\end{equation}
Note that we could have applied this simplification procedure to our earlier result~\eqref{eq:gens_Lie_only} and obtain the same generators.
However, as we show in~\Cref{sec:examples} (see~\Cref{table:runtimes}), the size of expressions typically makes this computation much harder and even out of reach.


\section{Formal description of solution}
\label{sec:formal}


\subsection{Splitting with respect to inputs}\label{sec:sep-inputs}

The goal of this section is to prove~\Cref{thm:split_input} with the following rationale behind it.
While we have defined 
observability
for functions in parameters, states, and inputs (see~\Cref{def:identifiability2}), the theorem implies that an 
observable
function involving inputs is a function of input-free 
observable
functions and inputs themselves.
This, in particular, shows that the observable
functions inside $\CC(\bmu, \bx)$ as in~\Cref{def:field_of_ident} together with $\bu^{(\infty)}$
generate all the observable functions in the sense of~\Cref{def:identifiability2}.

\begin{theorem}\label{thm:split_input}
    Let $h \in \CC(\bmu, \bx, \bu^{(\infty)})$ be an 
    observable
    function.
    We write $h = \frac{A}{B}$, where $A, B \in \CC(\bmu, \bx)[\bu^{(\infty)}]$ are coprime and at least one of them has a coefficient with respect to $\bu^{(\infty)}$ equal to one.

    Then all the coefficients of $A$ and $B$ viewed as polynomials in $\bu^{(\infty)}$ are 
    observable.
\end{theorem}

\begin{example}
    Consider the second Lie derivative of the output of the system~\eqref{eq:running} from~\Cref{sec:informal} (see~\eqref{eq:gens_Lie_only}):
    \[
    \cL_\Sigma^2(g) = \mu_4 \frac{(\mu_1 x + \mu_2)^2}{(\mu_3 + u)^2} (2 \mu_1 (\mu_1 x + \mu_2)-u') = \frac{ \boxed{2 \mu_1 \mu_4(\mu_1 x + \mu_2)^3} - \boxed{\mu_4(\mu_1 x + \mu_2)^2} u'}{\boxed{\mu_3^2} + \boxed{2\mu_3} u + \boxed{1}\cdot u^2}.
    \]
    The coefficients with respect to $u^{(\infty)}$ are boxed in the equation above.
    One of them (in front of $u^2$) is equal to one, so~\Cref{thm:split_input} implies that all the other coefficients,  $2 \mu_1\mu_4(\mu_1 x + \mu_2)^3, \mu_4(\mu_1 x + \mu_2)^2, \mu_3^2, 2\mu_3$ are observable.
\end{example}

In order to prove the theorem, we will first reformulate the definition of identifiability in terms of abstract differential fields.

\begin{notation}[Differential field of an ODE system]\label{not:diff_field}
    Consider an ODE model as in~\eqref{eq:ODEmodel},
    \[
    \bx' = \bF(\bmu, \bx, \bu),\quad \by = \bg(\bmu, \bx, \bu).
    \]
    We will construct \emph{the differential field of the system} as follows.
    We start with the
    differential field $\CC(\bM)$, where $\bM$ are constants (i.e., $\bM' = 0$) independent over $\CC$ with $|\bM| = |\bmu|$, where $|\bM|$ denotes the number of elements in $\bM$.
    Next, we adjoin infinitely many tuples of independent elements $\bU, \bU', \bU'', \ldots$, each of size $|\bu|$, and define the derivation on $\CC(\bM, \bU^{(\infty)})$ by $(\bU^{(i)})' = \bU^{(i + 1)}$ for $i \geqslant 0$. 
    Finally, we adjoin $|\bx|$ independent variables $\bX$ and define the derivation on $\CC(\bM, \bX, \bU^{(\infty)})$ by 
    \[
    \bX' = \bF(\bM, \bX, \bU).
    \]
    The resulting differential field $\CC(\bM, \bX, \bU^{(\infty)})$ contains elements $\bY$ defined by $\bY = \bg(\bM, \bX, \bU^{(\infty)})$.
\end{notation}

\begin{lemma}\label{lem:generic_sol}
    Using~\Cref{not:diff_field}, a function $h(\bmu, \bx, \bu) \in \CC(\bmu, \bx, \bu^{(\infty)})$ is 
    observable
    if and only if, the following holds inside the differential field of the system:
    \[
    h(\bM, \bX, \bU) \in \CC(\bY^{(\infty)}, \bU^{(\infty)}).
    \]
\end{lemma}

\begin{remark}
    The elements $\bM, \bX, \bU, \bY$ of the differential field are sometimes referred as \emph{a generic solution} of the model (e.g., \cite[Definition~6]{Ovchinnikov21-2790}, \cite[Section II.6]{Ritt}).
    The characterization of observability given by~\Cref{lem:generic_sol}
    is often used as a definition of observability~\cite[Definition~7]{Ovchinnikov21-2790}.

    Combining~\Cref{lem:generic_sol} with \cite[Proposition~3.4]{Hong20} one can show that our~\Cref{def:identifiability2} is equivalent to the analytic definition of observability (e.g, \cite[Definition~2.3]{Hong20}).
\end{remark}

\begin{lemma}\label{lem:field_to_func}
    Let $P(\bmu, \bx, \by, \bu) \in \CC[\bmu, \bx, \by^{(\infty)}, \bu^{(\infty)}]$ be a polynomial.
    Then $P(\bM, \bX, \bY, \bU) = 0$ if and only if $P(\bmu^*, \bx^*, \by^*, \bu^*) = 0$ for any analytic solution $\bmu^*, \bx^*, \by^*, \bu^*$ of~\eqref{eq:ODEmodel}.
\end{lemma}

\begin{proof}
    Assume that $P(\bM, \bX, \bY, \bU) \neq 0$.
    Let $\cF$ be the differential subfield of $\CC(\bM, \bX, \bU^{(\infty)})$ generated by $\bM, \bX, \bU^{(\infty)}$ and the coefficients of $P, \bF, \bg$ over $\mathbb{C}$.
    By Seidenberg embedding theorem~\cite{Seidenberg1958,Pavlov2022}, this field can be embedded into the field of meromorphic function on some domain.
    Since $\bX' = \bF(\bM, \bX, \bU)$ and $\bY = \bg(\bM, \bX, \bU)$, the images of $\bM, \bX, \bY, \bU$ under this embedding constitute an analytic solution of~\eqref{eq:ODEmodel}.
    Since $\cF$ contains $P(\bM, \bX, \bY, \bU) \neq 0$, this solution does not vanish $P$ either.

    Assume that $P(\bM, \bX, \bU, \bY) = 0$ and let $\bmu^\ast, \bx^\ast, \by^\ast, \bu^\ast$ be an analytic solution of~\eqref{eq:ODEmodel} 
    at which $P$ does not vanish.
    By differentiating $\bg(\bM, \bX, \bU)$ and eliminating $\bX'$ using $\bX' = \bF(\bM, \bX, \bU)$, one can show by induction that, for every $i \geqslant 0$, there exists $\bg_i \in \CC(\bmu, \bx, \bu^{(\infty)})$ such that $\bY^{(i)} = \bg_i(\bM, \bX, \bU)$ (where $\bg_0 = \bg$).
    Since every analytic solution satisfies~\eqref{eq:ODEmodel},
    we also have $(\by^*)^{(i)} = \bg_i(\bmu^*, \bx^*, \bu^*)$.
    Let $Q \in \CC(\bmu, \bx, \bu^{(\infty)})$ be the rational function obtained from $P$ by replacing every $\by^{(i)}$ by $\bg_i$.
    On one hand, $Q$ is a nonzero polynomial since it does not vanish at $\mu^\ast, \bx^\ast, \bu^\ast$.
    On the other hand, $Q(\bM, \bX, \bU) = 0$ which contradicts the algebraic independence of $\bM, \bX, \bU^{(\infty)}$.
    The obtained contradiction proves that $P(\bM, \bX, \bY, \bU) = 0$ implies the same equality for any analytic solution of~\eqref{eq:ODEmodel}.  
\end{proof}

\begin{proof}[Proof of~\Cref{lem:generic_sol}]
    Assume that function $h$ is identifiable, and let $P_1, P_2 \in \CC[\by^{(\infty)}, \bu^{(\infty)}]$ be the corresponding differential polynomials from~\Cref{def:identifiability2}.
    Let $h = \frac{H_1}{H_2}$ for coprime $H_1, H_2 \in \CC[\bmu, \bx, \bu^{(\infty)}]$.
    \Cref{lem:field_to_func} implies that $P_2(\bY, \bU) \neq 0$, so we can consider
    \[
    h(\bM, \bX, \bU) - \frac{P_1(\bY, \bU)}{P_2(\bY, \bU)} = \frac{H_1(\bM, \bX, \bU) P_2(\bY, \bU) - H_2(\bM, \bX, \bU) P_1(\bY, \bU)}{H_2(\bM, \bX, \bU) P_2(\bY, \bU)}.
    \]
    Identifiability of $h$ implies that the product of the numerator and denominator of the latter expression vanishes for any analytic solution of~\eqref{eq:ODEmodel}.
    Then, by~\Cref{lem:field_to_func}, it is equal to zero in the differential field.
    Since the denominator is nonzero in the field, the numerator vanishes implying that $h(\bM, \bX, \bU) \in \CC(\bY^{(\infty)}, \bU^{(\infty)})$.

    Assume $h(\bM, \bX, \bU) \in \CC(\bY, \bU)$.
    Then there exist coprime $P_1, P_2 \in \CC[\by^{(\infty)}, \bu^{(\infty)}]$ such that $P_2(\bY, \bU) \neq 0$ and $h(\bM, \bX, \bU) = \frac{P_1(\bY, \bU)}{P_2(\bY, \bU)}$.
    By~\Cref{lem:field_to_func}, the numerator of $h(\bmu, \bx, \bu) - \frac{P_1(\by, \bu)}{P_2(\by, \bu)}$ vanishes on any analytic solution of~\eqref{eq:ODEmodel}.
    This implies identifiability of $h$.
\end{proof}

In this field-theoretic reformulation, \Cref{thm:split_input} follows directly from the following proposition.
The proposition is a general statement about differential fields of special form, so we do not use~\Cref{not:diff_field}.

\begin{proposition}\label{prop:input_separation}
    Let $k$ be a differential field of characteristic zero\footnote{In our applications, we will have $k = \CC{(\bmu)}$} and consider the field $\cF_0$ of differential rational functions $k(\bu^{(\infty)})$.
    Let $\cF_2 = k(\bx, \bu^{(\infty)})$ be a transcendental extension by $\bx$ with the derivation extended by
    \[
        \bx' = \bF(\bx, \bu)
    \]
    for some $\bF$, a tuple of elements from $k(\bx, \bu)$.
    Let $\cF_1$ be any differential field with $\cF_0 \subseteq \cF_1 \subseteq \cF_2$. 
    Let $p \in \cF_1$ be an element, and we write it as a differential rational function in $\bu$, that is,
    \[
      p = \frac{A(\bu)}{B(\bu)},
    \]
    where $A, B \in k(\bx)[\bu^{(\infty)}]$ are coprime and for at least one of $A$ and $B$ at least one of their coefficients is equal to~$1$.
    Then all the coefficients of $A$ and $B$ belong to $\cF_1$.
\end{proposition}

The proof of~\Cref{prop:input_separation} will be based on two lemmas. 
\begin{notation}
  For $\bx = (x_1,\ldots,x_n)$ and a differential polynomial $P$ from $k[\bx^{(\infty)}]$, define~\emph{the order} of $P$ with respect to $x_i$, denoted by $\operatorname{ord}_{x_i} P$, to be the largest $j$ such that $x_i^{(j)}$ appears in $P$.
  If none of the derivatives of $x_i$ appears in $P$, we set $\operatorname{ord}_{x_i} P = -1$.
\end{notation}
\begin{lemma}\label{lem:separant}
    In the notation of~\Cref{prop:input_separation}, assume that $\ord_{u_1}p = h$.
    Then $\frac{\partial}{\partial u_1^{(h)}} (p) \in \cF_1$.
\end{lemma}

\begin{proof}
    Since $\cF_1$ is a subfield of a finitely generated extension of $\cF_0$, there exist $q_1, \ldots, q_\ell$ such that $\cF_1 = \cF_0(q_1, \ldots, q_\ell)$.
    Let $H$ be the maximum of the orders of $p, q_1, \ldots, q_\ell$ with respect to $u_1$.
    Since $\cF_1$ is a differential field, $p^{(H - h + 1)} \in \cF_1$.
    We have
    \[
      p' = u_1^{(h + 1)} \underbrace{\frac{\partial}{\partial u_1^{(h)}} (p)}_{\ord_{u_1} \leqslant h} + \underbrace{Q_1}_{\ord_{u_1} \leqslant h} \quad \implies \quad p^{(H - h + 1)} = u_1^{(H + 1)} \underbrace{\frac{\partial}{\partial u_1^{(h)}} (p)}_{\ord_{u_1} \leqslant H} + \underbrace{Q_2}_{\ord_{u_1} \leqslant H}.
    \]
    We write $p^{(H - h + 1)} = F(q_1, \ldots, q_\ell)$, where $F \in \cF_0(q_1, \ldots, q_\ell)$.
    Then we take an equality
    \begin{equation}\label{eq:before_shift}
      u_1^{(H + 1)} \frac{\partial}{\partial u_1^{(h)}} (p) + Q_2 = F(q_1, \ldots, q_\ell)
    \end{equation}
    and, considering it as an equality of rational functions in $\bu^{(\infty)}$ over $k(\bx)$, perform a substitution $u_1^{(H + 1)} \to u_1^{(H + 1)} + 1$.
    Since the orders of $Q_{2}, q_1, \ldots, q_\ell$ in $u_1$ do not exceed $H$, they will not be affected by the substitution, so we have
    \begin{equation}\label{eq:after_shift}
      (u_1^{(H + 1)} + 1) \frac{\partial}{\partial u_1^{(h)}} (p) + Q_{2} = \widetilde{F}(q_1, \ldots, q_\ell),
    \end{equation}
    where $\widetilde{F}$ is a result of applying the substitution to $F$.
    Subtracting~\eqref{eq:before_shift} from~\eqref{eq:after_shift}, we obtain
    \[
      \frac{\partial}{\partial u_1^{(h)}} (p) = \widetilde{F}(q_1, \ldots, q_\ell) - F(q_1, \ldots, q_\ell) \in \cF_1.\qedhere
    \]
\end{proof}

\begin{lemma}\label{lem:univariate}
    Let $\cK$ be a field of characteristic zero and $\cK(t)$ be the field of rational functions over $\cK$ equipped with the standard derivation $\frac{\partial}{\partial t}$.
    Let $p = \frac{a(t)}{b(t)} \in K(t)$ be a rational function such that $a(t)$ and $b(t)$ are coprime and $b(t)$ is monic.
    Then the field $\QQ(t, p^{(\infty)})$ contains all the coefficients of $a(t)$ and $b(t)$.
\end{lemma}

\begin{proof}
    We add an extra transcendental indeterminate $T$.
    Now we will view $p(t), p'(t), p''(t), \ldots$ as evaluations of the rational function $p(T)$ and its derivatives at point $T = t$.
    Let $d := \deg a + \deg b$. 
    Then the properties of the Pad\'e approximation~\cite[Corollary 5.21]{MCA} imply that the coefficients of $p(T)$ can be computed using the extended Euclidean algorithm in the field $\QQ(t, p(t), \ldots, p^{(d)}(t))$ and thus belong to this field.
\end{proof}

\begin{proof}[Proof of \Cref{prop:input_separation}]
    Let $\bu = (u_1, \ldots, u_s)$.
    For every $1 \leqslant i \leqslant s$, we denote by $h_i$ the order of $p$ w.r.t. $u_i$.
    We will prove the proposition by induction in $H \coloneqq h_1 + \ldots + h_s$.
    If $H = -s$, $p$ does not depend on $\bu$ and the statement of the proposition is true.
    Assume that $H > -s$ and consider any $i$ with $h_i \geqslant 0$. 
    Without loss of generality, we will take $i = 1$.

    Replacing $\frac{A}{B}$ with $\frac{B}{A}$ if necessary,
    we will assume that one of the coefficients of $B$ with respect to $\bu^{(\infty)}$ is equal to one.
    Let $B_0$ be a coefficient of $B$ as a polynomial in $u_1^{(h_1)}$ such that one of the coefficients of $B_0$ with respect to $\bu^{(\infty)}$ is equal to one.
    We denote the remaining coefficients of $B$ in $u^{(h_1)}$ by $B_1, \ldots, B_{\ell_B}$.
    The coefficients of $A$ in $u^{(h_1)}$ will be denoted $A_0, \ldots, A_{\ell_A}$.
    By applying~\Cref{lem:separant} to $p$ repeatedly, we find that $\frac{\partial^i}{\partial (u_1^{(h_1)})^i}(p) \in \cF_1$ for every $i \geqslant 0$.
    We apply~\Cref{lem:univariate} with $p = p$, $t = u_1^{(h_{1})}$, and $\cK = k(\bx, u_1^{(< h_1)}, u_2^{(\infty)}, \ldots, u_s^{(\infty)})$ and obtain that $\frac{B_1}{B_0}, \ldots, \frac{B_{\ell_B}}{B_0}, \frac{A_0}{B_0}, \ldots, \frac{A_{\ell_A}}{B_0} \in \cF_1$.
    
    Coprimality of $A$ and $B$ as polynomials in $\bu^{(\infty)}$ implies that $B_0, \ldots, B_{\ell_B}, A_0, \ldots, A_{\ell_A}$ are coprime as well.
    This implies that there exists a $\mathbb{Q}$-linear combination $C$ of $B_1, \ldots, B_{\ell_B}, A_0, \ldots, A_{\ell_A}$ coprime with $B_0$.
    Since $\ord_{u_1} C, \ord_{u_1} B_0 < h$, we can apply the induction hypothesis to $\frac{C}{B_0}$ and conclude that the coefficients of $B_0$ and $C$ belong to $\cF_1$, so $B_0 \in \cF_1$. 
    Therefore, $B_0, \ldots, B_{\ell_B}, A_0, \ldots, A_{\ell_A} \in \cF_1$ and, applying the induction hypothesis again, we conclude that their coefficients as polynomials in $\bu^{(\infty)}$, which coincide with the coefficients of $A$ and $B$ in $\bu^{(\infty)}$, belong to $\cF_1$.
\end{proof}

\begin{proof}[Proof of~\Cref{thm:split_input}]
    The theorem follows from~\Cref{prop:input_separation} applied to $\cF_0 = \CC(\bM, \bU^{(\infty)}), \cF_1 = \CC(\bM, \bY^{(\infty)}, \bU^{(\infty)}),$ and $\cF_2 = \CC(\bM, \bX, \bU^{(\infty)})$ and $p = H$.
\end{proof}


\subsection{Generating
all observable functions}

The goal of this section is to prove~\Cref{theorem:complete}. 
For a system~$\Sigma$ as defined in~\eqref{eq:ODEmodel}, we introduce \emph{the Lie derivative}: for $h \in k(\bmu, \bx, \bu^{(\infty)})$, we define it as
\begin{equation}\label{eq:lie}
    \cL_\Sigma(h) := \sum\limits_{i = 1}^n f_i \frac{\partial}{\partial x_i}(h) + \sum\limits_{j = 1}^\ell\sum\limits_{k = 0}^\infty u_j^{(k + 1)} \frac{\partial }{\partial u_j^{(k)}} (h).
\end{equation}

\begin{notation}
Let $\bp = (p_1,\ldots,p_m)$ with $p_1,\ldots,p_m \in \mathbb{C}(\bmu, \bx, \bu^{(\infty)})$. Let $\bh = (h_1,\ldots,h_m)$ with $h_1,\ldots,h_m \in \ZZpos$.
\begin{itemize}
    \item We denote $\cL_\Sigma^{< \bh}(\bp) = (p_1,\cL_\Sigma(p_1),\ldots,\cL_\Sigma^{h_1-1}(p_1),\ldots,p_m,\cL_\Sigma(p_m),\ldots,\cL_\Sigma^{h_m-1}(p_m))$.
    \item For every $i\in\mathbb{N}$, we denote $\bh + i = (h_1+i,\ldots,h_m+i)$. We define $\cL_\Sigma^{\leqslant \bh}(\bp) = \cL_\Sigma^{< \bh + 1}(\bp)$.
\end{itemize}
\end{notation}

\begin{theorem}\label{theorem:complete}
Consider the system $\Sigma$ as in~\eqref{eq:ODEmodel} with the outputs $\by = (y_1,\ldots,y_m)$ with $y_i = g_i(\bmu,\bx,\bu)$ for $g_i \in \mathbb{C}(\bmu,\bx,\bu)$ for $i=1,\ldots,m$. 
Let $\cE \subseteq \CC(\bmu, \bx)$ be a subfield of the observation field 
closed under taking Lie derivative.
We write $\cE = \CC(\bc)$ with $\bc = (c_1,\ldots,c_p) \in \CC(\bmu, \bx)^p$. 

Let $h_1,\ldots,h_m \in \ZZpos$ be such that for every $i=1,\ldots,m$ 
the variables $\by, \bu$ satisfy a differential equation over $\cE$ of order $h_i+1$ in $y_i$ and of order at most $h_j$ in $y_j$ for $j\neq i$ for $j=1,\ldots,m$. 
We denote $\bh := (h_1,\ldots,h_m)$.
Then $h \in \CC(\bmu,\bx,\bu^{(\infty)})$ is observable if and only if
\[
h \in  \CC(\bc, \cL_\Sigma^{\leqslant \bh+1}(\bg),\bu^{(\infty)}).
\]
\end{theorem}

The following definition
can be seen as a generalization of observability indices of the system~\cite[\S 2.1 and \S 4.6]{Moog}. 

\begin{definition}[{Parametric profile, cf.~\cite[Definition 2.8]{Dong2023}}]\label{def:param-prof}
For a system $\Sigma$ as in~\eqref{eq:ODEmodel} and a subfield $\cE \subseteq \CC(\bmu,\bx)$ such that $\mathcal{L}_{\Sigma}(\cE) \subseteq \cE$, a tuple $\bh = (h_1,\ldots,h_m) \in (\ZZpos)^m$ is called a \emph{parametric profile} for $\Sigma$ and $\cE$ if \begin{itemize}
    \item 
    $\by, \bu$ do not satisfy any nontrivial differential equation over $\cE$ of order at most $h_i$ in $y_i$ for every $i = 1, \ldots, m$, and
    \item for every $i=1,\ldots,m$, there exists a differential equation $Q_i$ over $\cE$ satisfied by 
    $\by, \bu$ such that $\ord_{y_i}Q_i = h_i + 1$ and $\ord_{y_j}Q_i \leqslant h_j$ for every $j \neq i$, $j = 1, \ldots, m$.
\end{itemize}
\end{definition}

\begin{example}[Dependence on the ground field $\cE$]\label{ex:paramprof}
Consider the model $x'(t) = \mu_1 x(t)$ and $y(t) = x(t)$. Naturally, $y(t)$ satisfies the differential equation $y' - \mu_1 y = 0$ over $\cE = \CC(\mu_1)$, hence $h = 0$ is a parametric profile for $\cE = \CC(\mu_1)$.
However, if we take $\cE = \CC$ (that is, prohibit $\mu_1$ to appear in the coefficients), then $y y'' - (y')^2 = 0$ is the equation of the smallest order satisfied by $y(t)$, and the parametric profile is $h=1$.
\end{example}

\begin{remark}
\Cref{theorem:complete} gives an upper bound on the required orders of Lie derivatives. 
There may exist a tuple $\bh^* \leqslant \bh + 1$, with $h^*_i < h_i+1$ for some $i = 1,\ldots,m$, such that the statement still holds with $\bh+1$ replaced by $\bh^*$. 
However, if additionally $\bh$ is a parametric profile for $\cE$, for any such $\bh^*$, definition of parametric profile implies (see also
\cite[Definition~2.8]{Dong2023}) that $\bh\leqslant\bh^*\leqslant\bh+1$.
\end{remark}

\begin{lemma}\label{lemma:one-after-another}
    Let $k$ be a field of characteristic zero.
    Let $k \subseteq \cE \subseteq \cF \subseteq k(\bmu, \bx, \bu^{(\infty)})$ be field extensions, such that $\cL_{\Sigma}(\cE) \subseteq \cE$ and $\cF = k(\cL_{\Sigma}^{(\infty)}(\bg), \bu^{(\infty)})$ for $\bg = (g_1,\ldots,g_m)$ with $g_1,\ldots,g_m\in k(\bmu, \bx, \bu)$.
    Let $\bh = (h_1,\ldots,h_m) \in (\ZZpos)^m$ be such that for every $i=1,\ldots,m$,
    $\cL_{\Sigma}^{h_i+1}(g_i)$ is algebraic over 
    $
    \cE(\cL_\Sigma^{\leqslant\bh}(\bg), \bu^{(\infty)})
    $.
    Then
    \[
    \cF = \cE(\cL_\Sigma^{\leqslant 
    \bh+1
    }(\bg), \bu^{(\infty)}).
    \]
\end{lemma} 

\begin{proof}
    By decrementing some $h_i$'s if necessary, 
    it follows from \cite[Ch.~8, \S 1, Theorem 1.1]{lang} that there exists a parametric profile $\bh^* = (h_1^*,\ldots,h_m^*)$ for $\Sigma$ and $\cE$ such that $h_i^* \leqslant h_i$ for $i=1,\ldots,m$.
    If the statement of the lemma holds in the case $\bh = \bh^*$, then it holds in all cases.
    Thus it suffices to prove the lemma in the case $\bh = \bh^*$. Assume $\bh = \bh^*$.

    Let $\mathcal{K}_{
    {\bh}} = \cE(\cL_\Sigma^{
    \leqslant \bh
    }(\bg), \bu^{(\infty)})$.
    It suffices to show that
    \[
    \mathcal{K}_{{\bh+1}} = \mathcal{K}_{\bh+2} = \mathcal{K}_{{\bh+3}} = \ldots.
    \]
    By induction, it suffices to show that $\mathcal{K}_{{\bh+1}} = \mathcal{K}_{\bh+2}$.

    Let $i \in \{1,\ldots,m\}$ be arbitrary. Per hypothesis, $\cL^{
    {h_i+1}}_\Sigma(g_i)$ is algebraic over $\mathcal{K}_{
    {\bh}}$. Choose $Q \in \mathcal{K}_{
    {\bh}}[T]$ of minimal degree in $T$ (and this degree is not zero due to $\bh = \bh^*$) such that
    \[
    Q(\cL_\Sigma^{
    {h_i+1}}(g_i)) = 0.
    \]
    Let $H>0$ be arbitrary. By $H$ applications of $\cL_\Sigma$ to this equality and using the chain rule, we obtain
    \[
    \cL^{
    {h_i+1+H}}_\Sigma(g_i) \left( \frac{\partial}{\partial T} Q \right)(\cL_\Sigma^{
    {h_i+1}}(g_i)) + R(\cL_\Sigma^{
    {h_i+1}}(g_i),\ldots,\cL_\Sigma^{
    {h_i+H}}(g_i)) = 0,
    \]
    where $R \in \mathcal{K}_{{\bh+H}}[T_0,\ldots,T_{H-1}]$. Therefore,
    \begin{equation}\label{eq:express_L_H-2}
    \cL_\Sigma^{
    {h_i+1+H}}(g_i) = \frac{-R(\cL_\Sigma^{
    {h_i+1}}(g_i),\ldots,\cL_\Sigma^{
    {h_i+H}}(g_i))}{(\frac{\partial}{\partial T} Q) (\cL_\Sigma^{
    {h_i+1}}(g_i))},
    \end{equation}
    where the denominator is not zero: since $\operatorname{char} k = 0$, $\frac{\partial Q}{\partial T}$ is not identically zero, so if it vanished at the point, this would contradict minimality of the degree of $Q$ in $T$. 
    Both the numerator and the denominator on the right-hand side in~\eqref{eq:express_L_H-2} are polynomials in $\cL_\Sigma^{
    {h_i+1}}(g_i),\ldots,\cL_\Sigma^{
    {h_i+H}}(g_i)$ with coefficients in $\mathcal{K}_{
    {\bh+H}}$, and thus both belong to $\mathcal{K}_{
    {\bh+H}}$. Therefore, $\cL_\Sigma^{
    {h_i+1+H}}(g_i) \in \mathcal{K}_{
    {\bh+H}}$.

    We have shown that for every $H>0$ we have $\cL_\Sigma^{
    {\bh+1+H}}(\bg) \in \cE(\cL_\Sigma^{\leqslant
    {\bh+H}}(\bg), \bu^{(\infty)})$.
    In particular, $\mathcal{K}_{
    {\bh+1}} = \mathcal{K}_{
    {\bh+2}}$, and the claim follows.
\end{proof}

\begin{proof}[Proof of~\Cref{theorem:complete}]
    Subfield $\cE$ and tuple $\bh$ satisfy the assumptions of~\Cref{lemma:one-after-another}.
    We apply~\Cref{lemma:one-after-another} to $\cF = \CC(\cL_\Sigma^{(\infty)}(\bg), \bu^{(\infty)})$, where $\by = \bg(\bmu, \bx, \bu)$ are as in~\eqref{eq:ODEmodel}, which yields
    \[
    \cF = \cE(\cL_\Sigma^{\leqslant {\bh+1}}(\bg), \bu^{(\infty)}).
    \]
    Consider the differential field $\CC(\bM,\bX,\bU^{(\infty)})$ of the system $\Sigma$ defined using~\Cref{not:diff_field}.
    By construction, for every $i=1,\ldots,m$ and $h \geqslant 0$ repeated differentiation of $g_i(\bM,\bX, \bU)$ yields the equality in this field:
    \[
    Y_i^{(h)} = \cL_\Sigma^{h}(g_i)(\bM, \bX, \bU). 
    \]
    By~\Cref{lem:generic_sol},
    $h\in\CC(\bmu,\bx,\bu^{(\infty)})$ is observable if and only if
    \[
    h(\bM,\bX,\bU)
    \in
    \CC(\bY^{(\infty)},\bU^{(\infty)}).
    \]
    Therefore, $h$ is observable if and only if $h\in\cF$, and the claim follows.
\end{proof}

\subsection{
Where identifiability and observability meet}\label{sec:io-and-lie}

We combine the results of the previous sections to describe how to obtain generators of the observation field
of a given model. As before, we are considering a system $\Sigma$ as in~\eqref{eq:ODEmodel}, given by
\[
\bx' = \bF(\bmu, \bx, \bu),\quad \by = \bg(\bmu, \bx, \bu),
\]
with parameters $\bmu$, state variables $\bx$, input variables $\bu$, and output variables $\by$.

While the Lie derivatives of the outputs generate the observation field (e.g., by~\Cref{theorem:complete}), the required orders may be quite high (see~\Cref{sec:implementation}).
Roughly speaking, the following corollary shows that, if we start with a set of observable functions of parameters, the required orders become lower (controlled by a parametric profile, see \Cref{ex:paramprof}).

\begin{corollary}\label{cor:main}
Let $\cE \subseteq \CC(\bmu)$ be a
field whose elements are observable.
We write $\cE = \CC(c_1,\ldots,c_p)$ with $c_1,\ldots,c_p \in \CC(\bmu)$.
Let $\bh \in (\ZZpos)^m$ be a parametric profile for $\Sigma$ and $\cE$. 
Then, the observation field of $\Sigma$ is generated by $c_1,\ldots,c_p$ and by the coefficients of $\bg, \ldots,\cL_\Sigma^{\bh+1}(\bg)$ in the sense of~\Cref{thm:split_input}.
\end{corollary}

\begin{proof}
    Let $q_1,\ldots,q_\eta \in \CC(\bmu,\bx)$ denote the coefficients of $\bg, \ldots,\cL_\Sigma^{\bh+1}(\bg)$ in the sense of~\Cref{thm:split_input}.
    It suffices to show that
    \[
    h \in \CC(\bmu, \bx)~\text{is observable} ~\Longleftrightarrow~h\in \CC(c_1,\ldots,c_p,q_1,\ldots,q_\eta).
    \]

    Assume $h \in \CC(\bmu,\bx)$ is observable. Let $\bc =(c_1,\ldots,c_p)$ and $\bq = (q_1,\ldots,q_\eta)$. It suffices to show that $h \in \CC(\bc,\bq)$. 
    Voil\`a:
    \[
    \begin{aligned}
    && h & \in &\CC(\bc, \cL_\Sigma^{\leqslant \bh+1}(\bg), \bu^{(\infty)})&&\text{(by~\Cref{theorem:complete})}\\
    \Leftrightarrow && h & \in &\CC(\bc, \bq, \bu^{(\infty)})&&\text{(since $\cL_\Sigma^{\leqslant \bh+1}(\bg) \subseteq \CC(\bq,\bu^{(\infty)})$)}\\
    \Leftrightarrow && h & \in &\CC(\bc, \bq)&&\text{}\qedhere\\
    \end{aligned}
    \]
\end{proof}

\begin{remark}[Prolongations only approach]\label{remark:prolongations-only}
 The observation field can be generated by the coefficients of $\bg, \ldots,\cL_\Sigma^{|\bx|+|\bmu|}(\bg)$ in the sense of~\Cref{thm:split_input} alone, which follows from \Cref{cor:main} with $\cE := \CC$.
\end{remark}

\begin{remark}[Encoding parameters]\label{remark:parameters}
    Given a model $\Sigma$ with parameters, one can convert it to an equivalent model without parameters by making every parameter a state variable satisfying the equation $x'(t) = 0$. 
    Then, \Cref{cor:main} reduces to the approach from~\Cref{remark:prolongations-only}.
\end{remark}

\begin{example}
    Consider a model of the DC motor~\cite[Chapter VI, Exercise 6.28]{distefano2012feedback}:
    \[
    \left\{\begin{aligned}
    L i'(t) &= -R i(t) - K_e \omega(t) + u(t),\\
    J \omega'(t) &= K_t i(t) - B \omega(t),
\end{aligned}\right.\quad\quad y(t) = \omega(t).
    \]
    Here, $\bx = (i(t), \omega(t))$ and $\bmu = (R, L, K_e, K_t, J, B)$. The output is $y(t)$ and $u(t)$ is the input.
    
    If $\Omega(s), U(s)$ are the Laplace transforms of $y(t)$ and $u(t)$, respectively, then the transfer function~\cite[Chapter VI]{distefano2012feedback} of the system is
    \[
    \frac{\Omega(s)}{U(s)} =
    \frac{1}
    {c_1 s^2
    +c_2 s
    +c_3},
    \]
    where \[(c_1,c_2,c_3) = \left(\frac{J L}{K_t}, \frac{J R + B L}{K_t}, \frac{B R + K_e K_t}{K_t}\right).\]
    Then $c_1,c_2,c_3$ are identifiable since for linear models the coefficients of the transfer function coincide with the coefficients of the input-output equation and also due to~\cite[Theorem 1]{linear-io}.
    Equivalently, we could have obtained these coefficients from the input-output equation itself by eliminating the states.

    We take $\cE := \CC(c_1,c_2,c_3)$. 
    To apply~\Cref{cor:main} we also need to know a parametric profile $h$ for $\cE$ (as in~\Cref{def:param-prof}) or an upper bound for it. 
    Note that 
    $h \leqslant 1$ since $y(t)$ and $u(t)$ satisfy a differential equation over $\cE$ of order two. 
    It follows from~\Cref{cor:main} that the 
    observation
    field is generated by $c_1, c_2, c_3$ together with the coefficients of 
    $\omega, \cL_\Sigma (\omega), \cL_\Sigma^2 (\omega)$
    in terms of $u$.
    On the other hand, with the straightforward method (that applies \Cref{cor:main} with $\cE = \mathbb{C}$, as in \Cref{remark:prolongations-only}) one needs to 
    take derivatives up to order $|\bx| + |\bmu| = 8$.
\end{example}

\subsection{Main algorithm}

The corollary above yields the following approach to computing the observation field:

\begin{algorithm}[H]
\caption{Computing the observation field}
\label{alg:find_and_simplify}
\begin{description}[itemsep=0pt]
\item[Input:] a system $\Sigma$, given by $\bx' = \mathbf{f}(\bmu, \bx, \bu),~ \by = \bg(\bmu, \bx, \bu)$, as in~\eqref{eq:ODEmodel}.
\item[Output:] 
elements of $\mathbb{C}(\bmu, \bx)$
generating the observation field of $\Sigma$.
\end{description}

\begin{enumerate}[label = \textbf{(Step~\arabic*)}, leftmargin=*, align=left, labelsep=2pt, itemsep=4pt]
    \item \label{step:1} Compute some observable functions of parameters ($c_1,\ldots,c_p \in \CC(\bmu)$ in~\Cref{cor:main}):
    \begin{enumerate}[ref= \theenumi (\alph*)]
    \item\label{step:eliminate} Compute a parametric profile $\bh = (h_1,\ldots,h_m)$ for $\Sigma$ and $\CC(\bmu)$ and the corresponding differential equations $Q_1,\ldots,Q_m$ 
    as in~\Cref{def:param-prof}, for instance, using differential elimination~\cite[Algorithm 4.5]{Dong2023}.

    \item\label{step:identi_functions} Normalize each of $Q_1,\ldots,Q_m$ so that each has at least one coefficient with respect to $\by, \bu$ equal to one. Let $c_1,\ldots,c_p \in \CC(\bmu)$ be the coefficients of $Q_1,\ldots,Q_m$ with respect to $\by, \bu$.
    
    \item \label{step:test} Check if $c_1,\ldots,c_p$ are observable, for instance, using \cite[Algorithm~5.2]{Dong2023} (see also~\cite[Lemma 1]{Ovchinnikov21}). 
    If they are not, then convert parameters into state variables as in~\Cref{remark:parameters}, and 
    return the output of~\Cref{alg:find_and_simplify} applied to the modified model.

    \end{enumerate}
    \item \label{step:lie} Compute $q_1,\ldots,q_\eta \in \CC(\bmu,\bx)$, the coefficients of the Lie derivatives $\bg, \ldots,\cL_\Sigma^{\bh+1}(\bg)$ in the sense of~\Cref{thm:split_input}, by repeatedly differentiating the outputs $\bg$ symbolically.
    \item\label{step:simplify} Return a simple generating set of the field generated by $c_1,\ldots,c_p,q_1,\ldots,q_\eta$ by applying the simplification algorithm 
    \cite[Algorithm~8]{demin2026simplegeneratorsrationalfunction} (alternatively, use~\cite[Steps~1--3 of the Algorithm in Section~5]{Meshkat11}).
\end{enumerate}
\end{algorithm}

\begin{remark}[Early termination]\label{remark:early}
    In~\ref{step:lie}, one may alternate between taking Lie derivatives, extracting their coefficients using~\Cref{thm:split_input}, and taking Lie derivatives of these coefficients. 
    Roughly speaking, this amounts to augmenting the model on the fly with outputs corresponding to these coefficients. 
    As a result, differentiation orders lower than $\bh+1$ might suffice. 
    By~\Cref{theorem:complete}, one may terminate once higher order derivatives are algebraic over
    the field generated so far; this can be tested, for example, using \cite[Theorem~2.2]{ehrenborg1993apolarity}.
\end{remark}

\begin{remark}[Observability of individual parameters and states]\label{rem:odin_v_pole}
A particular state $x_i$ or parameter $\mu_j$ is observable if and only if it is among the generators returned by~\Cref{alg:find_and_simplify}. This is due to the following property of the simplification procedure from~\cite{demin2026simplegeneratorsrationalfunction}: if the field generated by $g_1(\bz), \ldots, g_\ell(\bz) \in \QQ(\bz)$ with $\bz = (z_1, \ldots, z_m)$ contains a variable, say $z_1$, then the simplified set of generators will contain $z_1$.
\end{remark}

\begin{proposition}
\label{prop:owo}
    \Cref{alg:find_and_simplify} is correct.
\end{proposition}

\begin{proof}
    First consider the case when $c_1,\ldots,c_p$ are observable in~\ref{step:test}. Note $\bh$ is a parametric profile for $\Sigma$ and $\cE := \CC(c_1,\ldots,c_p)$. Then the observation
    field is generated by $c_1,\ldots,c_p$ and $q_1,\ldots,q_\eta$ due to~\Cref{cor:main}.

    Now consider the case when $c_1,\ldots,c_p$ are not observable in~\ref{step:test}, so~\Cref{alg:find_and_simplify} is applied again to the modified model. Consider this second application. Since the modified model has no parameters, $c_1,\ldots,c_p$ are elements of $\CC$, and are thus observable. Then the 
    observation
    field of the modified model is generated by $c_1,\ldots,c_p$ and $q_1,\ldots,q_\eta$ due to~\Cref{cor:main}. The 
    observation
    fields of the original and modified models coincide, and the claim follows.
\end{proof}

\begin{example}
    Consider a variant of the Lotka-Volterra model,
    \[
    \left\{\begin{aligned}
    x_1'(t) &= x_1(t) \cdot (\alpha + \beta x_2(t) - \kappa u(t)),\\
    x_2'(t) &= x_2(t) \cdot (-\delta + \gamma x_1(t)),
\end{aligned}\right.\quad\quad y(t) = x_1(t),
    \]
    where $u(t)$ is the input and $y(t)$ is the output. 
    We have $\bx = (x_1(t),x_2(t))$
    and $\bmu = (\alpha,\beta,\delta,\gamma,\kappa)$.

    In the first step, \Cref{alg:find_and_simplify}
    finds some observable functions of parameters. 
    By eliminating $x_1$ and $x_2$, it
    calculates that $y(t)$ and $u(t)$ satisfy a differential equation of order two:
    \[
    y y''
    -\left(y'\right)^2
    +c_1 y y'
    -c_2 y^2 y'
    +c_3 y^2 u'
    +c_4 y^3
    -c_5 y^2
    -c_6 y^3 u
    +c_7 y^2 u = 0
    \]
    where $(c_1,\ldots,c_7) = (\delta, \gamma, \kappa, \alpha \gamma, \alpha \delta, \gamma \kappa, \delta \kappa)$. 
    It follows that the parametric profile for $\CC(\bmu)$ is $h = 1$.
    It can be checked by~\cite[Algorithm~5.2]{Dong2023} that the coefficients $c_1,\ldots,c_7$ of this input-output equation are identifiable, which concludes \ref{step:1} of the algorithm.

    At \ref{step:lie}, the coefficients of $x_1, \mathcal{L}_\Sigma(x_1), \mathcal{L}_\Sigma^2(x_1)$
    in terms of $u, u'$ are computed. 
    This yields that $q_1,\ldots,q_7$ correspond to the underlined coefficients in:
   \begin{align*}
       x_1 &= \underline{x_1} \quad \quad \cL_\Sigma(x_1) = -\underline{(\kappa x_1)} u + \underline{(\alpha x_1 + \beta x_1 x_2)},\\[0.3em]
       \cL_\Sigma^2 (x_1) &=\underline{(\kappa^2 x_1)} u^2 - \underline{(2\alpha\kappa x_1+2\beta\kappa x_1 x_2)} u - \underline{(\kappa x_1)} u' + \underline{(\alpha^2x_1 +\left(2\alpha\beta-\beta\delta\right)x_1x_2 +\beta^2x_1x_2^2 +\beta\gamma x_1^2x_2)}
   \end{align*}
    Using~\Cref{cor:main} with $\cE := \CC(c_1,\ldots,c_7)$, we may conclude that the observation field of the model is generated by $c_1,\ldots,c_7$ and $q_1,\ldots,q_7$, that is, $\cF_\Sigma = \CC(c_1,\ldots,c_7, q_1,\ldots, q_7)$.

    The final step is to apply algebraic simplification to $\cF_\Sigma$. Using \cite[Algorithm~8]{demin2026simplegeneratorsrationalfunction}, we obtain that $\cF_\Sigma$ can be generated simply by 
    \[
    \alpha, ~\gamma, ~\delta,~ \kappa, ~x_1, ~\beta x_2,
    \]
    which is the output of~\Cref{alg:find_and_simplify}.
\end{example}


\section{Implementation}\label{sec:implementation}

\Cref{alg:find_and_simplify} has been implemented in the 
\href{https://github.com/SciML/StructuralIdentifiability.jl}
{StructuralIdentifiability.jl}
package in Julia.
Since we use the algorithms from~\cite{demin2026simplegeneratorsrationalfunction}, our implementation is Monte-Carlo probabilistic in the sense that it may return an incorrect result but the
user can specify any positive upper bound on the probability of this event.
In this section, we will apply our implementation to several models (presented in detail in the next section) arising in applications with varying complexity and type of control and present some statistics illustrating the design choices made in~\Cref{alg:find_and_simplify}.
The computations were done using the version 0.5.21 of the software on MacBook Air with Intel i5 core (1,6 GHz) and 16 Gb RAM\footnote{The benchmark script is available at {\url{https://github.com/SciML/StructuralIdentifiability.jl/blob/v0.5.31/benchmarking/ObservableFunctions/benchmark_script.jl}}}.

First,
we will examine the impact of the identifiable function computation at~\ref{step:1} and of the final result simplification (which are optional as far as only the correctness of the output is concerned) on the complexity of the obtained generators.
To this end, we collect in \Cref{table:messiness} the highest degrees and term counts among the computed generators for various versions of the approach:
\begin{itemize}
    \item \emph{Prolongations only} (the combination of~\Cref{remark:prolongations-only,remark:early}) stands for using the approach described in the first two subsections of~\Cref{sec:informal}: the generators are obtained only by taking Lie derivatives and applying~\Cref{thm:split_input} (similarly to~\cite{Nemcova2016}).
    The \emph{orders} columns display the order of differentiation needed for every output.
    
    \item \emph{\ref{step:1} and prolongations} refers to performing the steps of~\Cref{alg:find_and_simplify} (and \emph{orders} column contains the parametric profile) but without the final simplification.
    \item \emph{\Cref{alg:find_and_simplify}} corresponds to the full algorithm. 
    We do not list the orders of the Lie derivatives since they coincide with the ones obtained in the previous item.
\end{itemize}

We observe that starting with identifiable functions of parameters allows to considerably reduce the 
orders
of Lie derivatives 
and, consequently, the complexity of the expressions.
Furthermore, the final simplification also allows to reduce the complexity even further and quite dramatically.
When considering examples one by one in~\Cref{sec:examples},
we will show that the resulting generators are not only compact but also often admit a domain-specific interpretation.

\begin{table}[htbp!]
    \centering
\begin{tabular}{l|c|c|c||c|c|c||c|c|}
\toprule
\multirow{2}{*}{Model} & \multicolumn{3}{c}{Prolongations only} & \multicolumn{3}{c}{\ref{step:1} and prolongations} & \multicolumn{2}{c}{\Cref{alg:find_and_simplify}}\\
\cmidrule(lr){2-4} \cmidrule(lr){5-7} \cmidrule(lr){8-9}
 & orders & degree & \# terms & orders & degree & \# terms & degree & \# terms \\
 
\midrule
\Cref{ex:siwr} & 7 & 14 & 313 & 3 &  6 &  15 & 3 & 3\\
\hline
\Cref{ex:cancer} ($y_2 = Q$) &  (5, 1) &  25 &  958 & (3, 1) &  13 & 94 & 2 & 2\\
\hline
\Cref{ex:cancer} ($y_2 = v$) & (5, 1) &  25 &  958 & (3, 1) &  13 & 94 & 3 & 2\\
\hline
\Cref{ex:enzyme} & 5 & 55 & 67 & 2 &  19  & 10 & 2 & 3 \\
\hline
\Cref{ex:SLIQR} & 6 & 21 & 162 & 4 & 14 & 41 & 3 & 4 \\
\hline
\Cref{ex:eaihrd} & 12 & 22 & 493813 & 6 & 10 & 748 & 5 & 7 \\
\bottomrule
\end{tabular}
    \caption{Degrees and sizes of generators of the observation field}
    \label{table:messiness}
\end{table}

While the final simplification could have been applied to the generators computed using \emph{Prolongations only} approach, \Cref{table:runtimes} shows that computing identifiable function of parameters first in~\ref{step:1} typically speeds up the computations (sometimes turning intractable problems into tractable ones!).

\begin{table}[htbp!]
    \centering
    \begin{tabular}{l|c|c|c|}
    \toprule
        \multirow{2}{*}{Model} & \multicolumn{2}{c}{Runtimes} & \multirow{2}{*}{Speedup} \\
\cmidrule(lr){2-3} 
 & \Cref{alg:find_and_simplify} without~\ref{step:1} & Full \Cref{alg:find_and_simplify} &  \\
 \midrule
        \Cref{ex:siwr} & 0.47 s. & 0.2 s. & 2.35 \\
         \hline
         \Cref{ex:cancer} ($y_2 = Q$) & 13.7 s. & 2.2 s. & 6.2 \\
         \hline
         \Cref{ex:cancer} ($y_2 = v$) & 161.2 s. & 234.6 s. & 0.69 \\
         \hline
         \Cref{ex:enzyme} & 29.6 s. & 0.56 s. & 52.9 \\
         \hline
         \Cref{ex:SLIQR} & 3.7 s. & 2.5 s. & 1.48 \\
         \hline
         \Cref{ex:eaihrd} & > 48 h. & 355.7 s. & $\infty$\\
         \bottomrule
    \end{tabular}
    \caption{Speedup achieved by~\ref{step:1}}
    \label{table:runtimes}
\end{table}

\section{Examples}
\label{sec:examples}

In this section, we apply our algorithm to several models from the literature and discuss interpretations of the obtained generators of the observation field.

\begin{example}[Cholera model]\label{ex:siwr}

We describe how our algorithm was applied in~\cite[Section 3.4]{Pant2026} to study identifiability of the following model describing the spread of cholera proposed in \cite{eisenberg2013identifiability}:

\begin{equation}\label{eq:SIRW}
\begin{cases}
S'(t) = - \beta_W S(t)W(t) - \beta_I S(t)I(t), \\
I'(t) = \beta_W S(t)W(t) + \beta_I S(t)I(t) - \gamma I(t), \\
W'(t) = \alpha I(t) - \zeta W(t), \\
y(t) = W(t).
\end{cases}
\end{equation}
In this model $S(t), I(t), W(t)$ are the state variables corresponding to susceptible and infected population and to the pathogen concentration in water, respectively.
Parameters are $\alpha, \beta_W, \beta_I, \gamma, \zeta$, and the observed quantity is the pathogen concentration $W$.
Our~\Cref{alg:find_and_simplify} finds the following generators for the observation field:
\begin{equation}\label{eq:siwr_result}
\frac{\alpha}{\beta_I}, \; \gamma + \zeta, \; \gamma \zeta, \; \frac{\alpha \beta_W + \beta_I \zeta}{\beta_I}, \; W(t), \; \beta_I S(t), \; \beta_I I(t) + \beta_W W(t).
\end{equation}
In this list, the first four quantities depending on parameters only are computed at~\ref{step:1} and the time-dependent identifiable functions are obtained by performing the rest of the computation.
By~\Cref{rem:odin_v_pole}, we see that none of $S(t), I(t), \alpha, \beta_I, \beta_W, \gamma, \zeta$ is observable.

We remark that the quantities~\eqref{eq:siwr_result} produced by the algorithm allow for domain-specific interpretation. 
For example, $\beta_I I(t) + \beta_W W(t)$ is the combined per-capita infection risk from human and contaminated water exposure (also called the force of infection) and $\beta_I S(t)$ reflects the fact that measuring only the pathogen concentration 
one cannot distinguish between large population with low transmission rate and small population with high transmission rate.
Furthermore, one can then use~\eqref{eq:siwr_result} to see that the effective reproduction number is identifiable.
More precisely, for this model this time-dependent quantity is equal to~$\frac{\alpha \beta_W + \beta_I \zeta}{\gamma\zeta} S(t)$ and it naturally decomposes as a product (see~\cite[Section~3.4]{Pant2026}):
\[
 \frac{\alpha \beta_W + \beta_I \zeta}{\gamma\zeta} S(t) = \frac{\alpha \beta_W + \beta_I \zeta}{\beta_I} \cdot \frac{1}{\gamma \zeta} \cdot \beta_I S(t).
\]
Other applications of~\eqref{eq:siwr_result} in~\cite{Pant2026} include determining possible complementary data sufficient to restore identifiability.
\end{example}


\begin{example}[Cancer modeling]\label{ex:cancer}
    We consider the following model of prostate cancer reproduced from~\cite[Eq. (12)-(15)]{Phan2023}, which incorporates the impact of  androgen on the tumor growth (for the history and variations of the model, see the references in~\cite{Phan2023}).
    \[
    \begin{cases}
        x'(t) = \mu_m \left(1 - \frac{q}{Q(t)}\right) x(t) - \left(v(t) \frac{R}{R + Q(t)} + \delta x(t)\right) x(t),\\
       v'(t) = -d v(t),\\
       Q'(t) = (\gamma_1 u(t) + \gamma_2) (Q_m - Q(t)) - \mu_m (Q(t) - q),\\
       P'(t) = b Q(t) + \sigma x(t) Q(t) - \varepsilon P(t).
    \end{cases}
    \]
    In this model, $x(t)$ is the cancer population, $Q(t)$ is the intracellular androgen level, $v(t)$ is the androgen-dependent death rate for the cancer cells, and $P(t)$ is the serum PSA (prostate-specific antigen) level.
    The model is controlled through the drug injection $u(t)$.

    The main measurement protocol considered in~\cite{Phan2023} involved two outputs $y_1(t) = P(t) , y_2(t) = Q(t)$.
    In this case, the observation field
    is generated by $\sigma x(t), \frac{\delta}{\sigma}$ and the remaining states and parameters.
    If one takes instead, for example, $y_2(t) = v(t)$, 
    a generating set of observation field 
    will be 
    \[
     bq, \; bR, \; b Q_m, \; b Q(t), \; \delta x(t), \; \frac{b\delta}{\sigma}
    \]
    and the remaining states and parameters ($\gamma_1, \gamma_2, \varepsilon, d, \mu_m, v(t), P(t)$).
    In both cases, the 
    observable
    functions are products of integer powers of parameters and states meaning that all the symmetries of the model are scaling transformations.
\end{example}


\begin{example}[Controlling enzyme kinetics]\label{ex:enzyme}
We consider a reaction with two competing substrates, $S_1(t)$ and $S_2(t)$, following the Michaelis-Menten kinetics (see, e.g~\cite[Section~3]{Schnell2000}) non-affinely controlled by $I(t)$, a common uncompetitve inhibitor~\cite[Eq.~(4)]{Hsu1979}.
By taking the total amount of the substrates as the output, we obtain:
\[
\begin{cases}
    S_1'(t) = \frac{V_1 S_1(t)}{1 + K_1 S_1(t) \left(1 + \frac{I(t)}{L_1}\right) + K_2 S_2(t) \left(1 + \frac{I(t)}{L_2}\right)},\\
    S_2'(t) = \frac{V_2 S_2(t)}{1 + K_1 S_1(t) \left(1 + \frac{I(t)}{L_1}\right) + K_2 S_2(t) \left(1 + \frac{I(t)}{L_2}\right)},\\
    y(t) = S_1(t) + S_2(t).
\end{cases}
\]
Applying~\Cref{alg:find_and_simplify}, we obtain the following generators of the observation field:
\begin{align*}
  &V_1 + V_2, & &L_1 + L_2, & &K_1 + K_2, & &S_1(t) + S_2(t), & &L_1 V_2 + L_2 V_1,\\
  &V_1 V_2, & &L_1 L_2, & &K_1 K_2, & &S_1(t) S_2(t), & &V_1 S_1(t) + V_2 S_2(t).
\end{align*}
The presence of sums and product in the list can be naturally explained by the symmetry of the model exchanging the substrates.
Interestingly, in addition to the observed sum $S_1(t) + S_2(t)$, we also see that the product $S_1(t) S_2(t)$ and the sum weighted by the limiting rates $V_1 S_1(t) + V_2 S_2(t)$ is among the generators.

\end{example}


\begin{example}[Modeling a disease with a relapse after remission]\label{ex:SLIQR}
    The following compartmental epidemiological model was considered in~\cite[Eq~(4)]{Dankwa2022} in order to take into account the possibility of a relapse after a period of remission:
    \[
    \begin{cases}
        S'(t) = -\beta \frac{S(t)I(t)}{N} - u(t) \frac{S(t)}{N},\\
        L'(t) = \beta \frac{S(t)I(t)}{N} - \alpha L(t),\\
        I'(t) = \alpha L(t) - \gamma I(t) + \sigma Q(t), \\
        Q'(t) = (1 - \nu) \gamma I(t) - \sigma Q(t),\\
        y(t) = \frac{I(t)}{N}.
    \end{cases}
    \]
    The states $S(t), L(t), I(t), Q(t)$ correspond to the susceptible, exposed, infectious individuals and the individuals at the remission stage, respectively.
    The model is controlled by the vaccination rate $u(t)$ which is assumed to be sufficiently generic.
    Computation with~\Cref{alg:find_and_simplify} returns the following 
    generating 
    observable
    functions:
    \[
      \sigma,\; \beta, \; N, \; \alpha + \gamma, \; \alpha \nu \gamma,\; \alpha \gamma + (\nu - 1) \gamma \sigma,\; I(t), \; \alpha S(t),\; (\alpha - \sigma)Q(t), \; \alpha L(t) - \gamma I(t) + \sigma Q(t).
    \]
    The observable functions involving state variables include the total number of infectious individuals ($I(t)$) as well as the rate of its change ($\alpha L(t) - \gamma I(t) + \sigma Q(t)$; can be thought as ``new daily cases'').
    The element $\alpha S(t)$ reflects the confounding of larger population with longer average latent period (which is equal to $\frac{1}{\alpha}$) and smaller population with shorter latent period.
    Similarly, $(\alpha - \sigma) Q(t)$ reflects a tradeoff between the number of individuals at the remission stage and the difference between the remission and latent periods.
\end{example}

\begin{example}[EAIHRD epidemiology model]\label{ex:eaihrd}
    Consider an epidemiological model from~\cite{Fokas2020} (see also~\cite{chen2024practicalidentifiabilityparameterestimation}) which takes
    into account multiple aspects such as hospitalization, mortality, and latent period.
    The model is given by the following ODE system (\cite[Eq. (1)-(6)]{Fokas2020} and~\cite[Eq. (6a)-(6f)]{chen2024practicalidentifiabilityparameterestimation}):
    \begin{equation}
        \begin{cases}
            A'(t) = a E(t) - r_1 A(t),\\
            I'(t) = s E(t) - (h + r_2) I(t),\\
            H'(t) = h I(t) - (r_3 + d) H(t),\\
            R'(t) = r_1 A(t) + r_2 I(t) + r_3 H(t),\\
            D'(t) = d H(t),\\
            E'(t) = (N - A(t) - I(t) - H(t) - R(t) - D(t) - E(t)) (c_1 A(t) + c_2 I(t)) - (a + s) E(t),\\
            y(t) = D(t).
        \end{cases}
    \end{equation}
    The states $A(t), I(t), H(t), R(t), D(t), E(t)$ correspond to asymptomatic, infected, hospitalized, recovered, dead, and exposed individuals.
    The observation function is taken to be the cumulative mortality.
    \Cref{alg:find_and_simplify} produces the following list of generators of the observation field:
    \begin{align*}
        &r_1,\quad d + r_3,\quad a + h + s + r_2,\quad ah + ar_2 + hs + sr_2,\quad \frac{dhs}{ac_1 + sc_2},\\
        &D(t), \quad d H(t),\quad dh I(t),\quad c_1 A(t) + c_2 I(t),\quad a + s + \frac{s E(t)}{I(t)},\\
        & c_2(h - r_1 + r_2) I(t) - (a c_1 + s c_2)E(t),\quad \frac{s (A(t) + H(t) +  R(t) +  D(t) - N)}{I(t)} - a.
    \end{align*}
    Several of the elements of the list admit natural interpretation. 
    For example, since only the cumulative mortality is observed, hospitalized and infected individuals appear in $d H(t)$ and $d h I(t)$, respectively, scaled by the mortality and hospitalization rates.
    The expression $c_1 A(t) + c_2 I(t)$ is the force of infection (i.e. combined per-capita infection risk, see~\Cref{ex:siwr}).
    Furthermore, the functions involving parameters only obtained by our algorithm are in an agreement with new parameters $F, R_2, R_3, C_1, C_2$ introduced by hand in~\cite[Eq. (9)]{Fokas2020} (see also~\cite[Section~2]{chen2024practicalidentifiabilityparameterestimation}).
    For example, $d + r_3 = R_3$, $a + h + r_2 + s = F + R_2$, $ah + ar_2 + hs + sr_2 = F R_2$, and $\frac{dhs}{ac_1 + sc_2}$ is the harmonic mean of $C_1$ and $C_2$. 
\end{example}

\bibliographystyle{abbrvnat}
\bibliography{refs}

\end{document}